%% file: main.tex
\documentclass[letterpaper, 10 pt, conference]{ieeeconf}  

\IEEEoverridecommandlockouts                              
\def\ARXIV{0}

\usepackage[style=ieee,backend=bibtex]{biblatex}
\usepackage{amsmath}
\usepackage{amssymb}
\usepackage{mathabx} 
\usepackage[scr=esstix]{mathalpha}
\usepackage{upgreek}
\usepackage{bm}
\usepackage{hyperref}
\hypersetup{colorlinks,
	linkcolor=blue,
	citecolor=blue,
	urlcolor=blue,
	linktocpage,
	plainpages=false}
\usepackage{cancel}
\usepackage{graphicx}
\usepackage{mathtools}
\usepackage{enumerate}
\usepackage{float}
\usepackage{tikz}
\usetikzlibrary{shapes,matrix,arrows,calc,positioning,fit,bayesnet,angles,patterns}

\tikzset{%
  every neuron/.style={
    circle,
    draw,
    minimum size=0.5cm
  },
  neuron missing/.style={
    draw=none, 
    scale=2,
    text height=0.333cm,
    execute at begin node=\color{black}$\vdots$
  },
}
\newtheorem{theorem}{Theorem}
\newtheorem{corollary}{Corollary}
\newtheorem{lemma}{Lemma}

\newtheorem{remark}{Remark}

\input{macros}

\begin{document}

\title{\LARGE \bf
  Approximation with Random Shallow ReLU Networks with Applications to
Model Reference Adaptive Control
}

\author{
Andrew Lamperski and 
  Tyler Lekang 
  \thanks{This work was supported in part by NSF CMMI-2122856}
  \thanks{A. Lamperski is with the department of Electrical and Computer Engineering, University  of Minnesota, Minneapolis, MN 55455, USA {\tt\small alampers@umn.edu}}
  \thanks{T. Lekang is with Honeywell Aerospace, Plymouth, MN 55441, USA
        {\tt\small tylerlekang@gmail.com}}
}

\maketitle
\thispagestyle{empty}
\pagestyle{empty}

\input{absIntroNotation}


\input{integral}


\input{application}

\input{conclusion}

\printbibliography

\if\ARXIV1
\newpage
\appendices
\onecolumn
\input{appendix}
\fi

\end{document}

%% file: macros.tex
\newcommand{\C}{\mathbb{C}}
\newcommand{\E}{\mathbb{E}}

\renewcommand{\P}{\mathbb{P}}

\newcommand{\R}{\mathbb{R}}
\renewcommand{\S}{\mathbb{S}}

\DeclareMathOperator{\Exp}{\E}

\newcommand{\cb}{\bm{c}}

\newcommand{\fb}{\bm{f}}

\newcommand{\vb}{\bm{v}}

\newcommand{\wb}{\bm{w}}

\newcommand{\xb}{\bm{x}} 

\newcommand{\tb}{\bm{t}}

\newcommand{\xib}{{\bm{\xi}}}

\newcommand{\Ab}{\bm{A}}

\newcommand{\thetab}{{\bm{\theta}}}
\newcommand{\alphab}{{\bm{\alpha}}}

\newcommand{\Thetab}{{\bm{\Theta}}}
\newcommand{\Psib}{{\bm{\Psi}}}

\newcommand{\Thetah}{{\widehat{\Theta}}}

\newcommand{\Jc}{\mathcal{J}}

\newcommand{\uw}{\widetilde{u}}


%% file: absIntroNotation.tex
\begin{abstract}
  Neural networks are regularly employed in adaptive control of nonlinear systems and related methods of reinforcement learning. A common architecture uses a neural network with a single hidden layer (i.e. a shallow network), in which the weights and biases are fixed in advance and only the output layer is trained. While classical results show that there exist neural networks of this type that can approximate arbitrary continuous functions over bounded regions, they are non-constructive, and the networks used in practice have no approximation guarantees. Thus, the approximation properties required for control with neural networks are assumed, rather than proved. In this paper, we aim to fill this gap by showing that for sufficiently smooth functions, ReLU networks with randomly generated weights and biases achieve $L_{\infty}$ error of $O(m^{-1/2})$ with high probability, where $m$ is the number of neurons. It suffices to generate the weights uniformly over a sphere and the biases uniformly over an interval. We show how the result can be used to get approximations of required accuracy in a model reference adaptive control application. 
 \end{abstract}

\section{Introduction}


Neural networks have wide applications in control systems, particularly for nonlinear systems with unknown dynamics. In adaptive control and, they are commonly used to model unknown nonlinearities \cite{lavretsky2013robust}. In reinforcement learning and dynamic programming, they are used to approximate value functions and to parameterize control strategies \cite{vrabie2013optimal,powell2007approximate,sutton2018reinforcement}.

In the control problems described above, the theoretical analysis requires that the deployed neural network can achieve a desired  level of approximation accuracy, particularly when approximating nonlinear dynamics in adaptive control and approximating value functions in reinforcement learning / approximate dynamic programming.
A typical goal is for  the controller to achieve good performance (such as low tracking error) when the state of the system is in a bounded set, and as a result, many applications require the neural networks provide accurate approximations on bounded sets.

A theoretical gap arises in the current use of neural networks in adaptive control and reinforcement learning, in that the required approximation properties are \emph{assumed} rather than proved \cite{vrabie2013optimal,kamalapurkar2018reinforcement,greene2023deep,makumi2023approximate,cohen2023safe,kokolakis2023reachability,sung2023robust,lian2024inverse}. See \cite{soudbakhsh2023data,annaswamy2023adaptive} for discussion. In particular, non-constructive classical results, reviewed in \cite{pinkus1999approximation}, show that neural networks exist which give the desired approximation accuracy, but do not give a practical approach to compute them. As a result, the neural networks deployed in practice lack approximation guarantees.

The main contribution of this paper shows that two-layer 
neural networks with ReLU activation functions and hidden-layer parameters generated on a compact set can approximate sufficiently smooth functions on bounded sets to arbitrary accuracy. For these networks, only the output layer is trainable. Such architectures are  common in control. It is shown that the  worst-case (i.e. $L_{\infty}$) error on balls around the origin decays like $O(m^{-1/2})$, where $m$ is the number of neurons. Any randomization scheme with strictly positive density will suffice, including the uniform distribution.

To prove our approximation theorem, we derive a new neural network integral representation theorem for ReLU activations and smooth functions over bounded domains, which may be of independent interest. Similar integral representations are commonly employed in constructive approximation theory for neural networks \cite{sonoda2017neural,petrosyan2020neural,irie1988capabilities,kainen2010integral}. The advantage of our new integral representation compared to the existing results is that the integrand can be precisely bounded. This bound enables the analysis of the randomized networks. 

As an application, we show how the approximation theorem can be used to construct neural networks that are suitable for model reference adaptive control problems with unknown nonlinearities. In particular, we show how quantify the number of neurons required to achieve the required accuracy for a control algorithm from \cite{lavretsky2013robust}.  


Over the last several years, the theoretical properties  of neural networks with random initializations have been studied extensively.
Well-known results show that as the width of a randomly initialized neural network increases, the behavior approaches a Gaussian process \cite{neal1994priors,lee2017deep,jacot2018neural}. Related work shows that with sufficient width, \cite{du2019gradient,allen2019convergence}, gradient descent reaches near global minima from random initializations.

The closest work on approximation is \cite{hsu2021approximation}, which also bounds the approximation error for random shallow ReLU networks. The key distinction between our result and \cite{hsu2021approximation} is that our error bound is substantially simpler and more explicit, which enables its use for adaptive control. (The error from \cite{hsu2021approximation} is a complex expression with unquantified constants.) Additionally, we bound the $L_{\infty}$ error, which is commonly required in control, while \cite{hsu2021approximation} bounds the $L_2$ error. The work in \cite{hsu2021approximation} has the advantage of applying to a broader class of functions, and also includes lower bounds that match the achievable approximation error. 

Other closely related research includes \cite{yehudai2019power,li2023powerful}. In \cite{yehudai2019power}, it is shown that learning a non-smooth function with a  randomized ReLU network requires a large number of neurons. (We approximate smooth functions in this paper.) Lower bounds on achievable errors for a different class of randomized single-hidden-layer networks are given in \cite{li2023powerful}.

Similarly motivated work by the authors includes \cite{lekang2022sufficient},  which gives sufficient conditions  for persistency of excitation of neural network approximators, and \cite{lamperski2022neural}, which shows that shallow neural networks with randomly generated weights and biases define linearly independent basis functions. Persistency of excitation and linear independence are commonly assumed without proof in the adaptive control literature to prove convergence of parameters and bound accuracy of approximation schemes, respectively.

The paper is organized as follows. Section~\ref{sec:notation}
presents preliminary notation. Section~\ref{sec:approximation} gives the main result on approximation. Section~\ref{sec:apx:application} presents an application to Model Reference Adaptive Control.
 We 
provide closing remarks in Section~\ref{sec:conclusion}.

\section{Notation}
\label{sec:notation}

We use $\R,\C$ to denote the real and complex numbers. Random variables are denoted as bold symbols, e.g. $\xb$. $\Exp[\xb]$ denotes the expected value of $\xb$ and $\P(\Ab)$ denotes the probability of event $\Ab$.
$B_x(R)\subset\R^n$ denotes the radius $R$ Euclidean ball centered at $x\in\R^n$.
Subscripts on vectors and matrices denote the row index, for example $W_i$ is the $i$th row
of $W$. The Euclidean norm is denoted
$\|w\|$, while if $M$ is a matrix, then $\|M\|$ denotes the induced $2$-norm.  If $f$ is a complex-valued function, and $p\in[1,\infty]$, $\|f\|_p$ denotes the corresponding $L_p$ norm. 
We use subscripts on
time-dependent variables to reduce parentheses, for example $x(t)$ is instead denoted
$x_t$.


%% file: integral.tex
\section{Approximation by Randomized ReLUs}
\label{sec:approximation}
This section gives our main technical result, which shows that all sufficiently smooth functions can be approximated by an affine function and a single-hidden-layer neural network with ReLU activations, where the weights and biases are generated randomly. In particular, we show that the worst-case error over a compact set decays like $O(m^{-1/2})$,  where $m$ is the number of neurons. 

\subsection{Background}
\label{ss:background}
If $f:\R^n \to \C$, it is related to its Fourier transform $\hat f:\R^n\to \C$ by
\begin{subequations}
\begin{align}
  \label{eq:ft}
  \hat f(\omega)&=\int_{\R^n}e^{-j2\pi \omega^\top x} f(x)dx\\
  \label{eq:ift}
  f(x)&=\int_{\R^n}e^{j2\pi \omega^\top x}\hat f(\omega)d\omega.
\end{align}
\end{subequations}

Let $\S^{n-1}=\{x\in\R^n|\|x\|=1\}$ denote the $n-1$-dimensional unit sphere.
We denote the area of the area of $\S^{n-1}$ by:
\begin{equation}
  \label{eq:sphereArea}
  A_{n-1} = \frac{2\pi^{n/2}}{\Gamma(n/2)}
\end{equation}
where $\Gamma$  is the gamma function. The area is maximized at $n=7$, and decreases geometrically with $n$.

Let $\mu_{n-1}$ denote the area measure over $\S^{n-1}$ so that $A_{n-1}=\int_{\S^{n-1}}\mu_{n-1}(d\alpha)$. In particular, $\S^{0}=\{-1,1\}$ and $\mu_{0}$ is the counting measure, with $\mu_{0}(\{-1\})=\mu_{0}(\{1\})=1$. 

Throughout the paper, $\sigma$ will denote the ReLU activation function:
\begin{equation}
  \label{eq:relu}
  \sigma(t)=\max\{0,t\}.
\end{equation}

\subsection{Approximation with Random ReLU Networks}

Our approximation result below holds for  functions $f:\R^n\to \R$ which satisfy the following smoothness condition:
\begin{equation}
  \label{eq:smoothness}
  \exists k\ge n+3, \rho > 0 \textrm{ such that }
  \sup_{\omega\in\R^n} |\hat f(\omega)| (1+\|\omega\|^k) \le \rho
\end{equation}
This condition implies, in particular, that $f$ and all of its derivatives up to order $k-2$ are bounded.  

Our main technical result is stated below:

\begin{theorem}
  \label{thm:approximation}
  {\it
    Let $P$ be a probability density function over $\S^{n-1}\times [-R,R]$ with $\inf_{(\alpha,t)\in \S^{n-1}\times [-R,R]}P(\alpha,t)=P_{\min}>0$. Let $(\alphab_1,\tb_1),\ldots,(\alphab_m,\tb_m)$ be independent, identically distributed samples from $P$. If $f$ satisfies \eqref{eq:smoothness}, then there is a vector $a\in \R^n$, a number $b\in \R$, and coefficients $\cb_1,\ldots,\cb_m$ with
    \begin{align*}
      \|a\| &\le 4\pi A_{n-1}\rho \\
      |b| & \le (1+(2\pi R)) A_{n-1}\rho \\
      |\cb_i| & \le \frac{8\pi^2 \rho}{m P_{\min}}
    \end{align*}
    such that for all $\nu \in (0,1)$, with probability at least $1-\nu$, the neural network approximation
    \begin{equation}
      \label{eq:fNN}
    \fb_N(x)=a^\top x +b + \sum_{i=1}^m\cb_i \sigma(\alphab_i^\top x-\tb_i)
    \end{equation}
    satisfies
    $$
    \sup_{x\in B_0(R)}|\fb_N(x)-f(x)|\le \frac{1}{\sqrt{m}}\left(\kappa_0+\kappa_1 \sqrt{\log(\nu/4)} \right).
    $$
    Here
    \begin{align*}
      \kappa_0 &=512n^{1/2} \pi^{5/2} R \rho \left(\frac{\pi}{P_{\min}}+A_{n-1} \right) \\
      \kappa_1 &=\frac{264\pi^2 \rho R}{P_{\min}}+\rho A_{n-1}\left(4+256 R\pi \right)
    \end{align*}
}
\end{theorem}

The special case of the uniform distribution has $P(\alpha,t)=\frac{1}{2RA_{n-1}}$ for all $(\alpha,t)\in\S^{n-1}\times  [-R,R]$, so that $\frac{1}{P_{\min}}=2RA_{n-1}$. In this case the bounds simplify:

\begin{corollary}
  \label{cor:uniform}
  {\it
    If $P$ is the uniform distribution over $\S^{n-1}\times [-R,R]$, the coefficients depending on $P_{\min}$ have the following upper bounds:
    \begin{align*}
      |\cb_i|&\le \frac{16\pi^2}{m}\rho A_{n-1} \\
      \kappa_0&\le 512n^{1/2} \pi^{5/2} R \rho A_{n-1} (1+2\pi R) \\
      \kappa_1&\le \rho A_{n-1} \left(528(\pi R)^2 + 256(\pi R) + 4\right).
    \end{align*}
    }
  \end{corollary}

  \begin{remark}
    The bounds on the coefficients depend on the product $\rho A_{n-1}$. While $A_{n-1}$ is decreasing with respect to $n$, the smoothness coefficient will typically be large for large $n$ (since $k\ge n+3$ in \eqref{eq:smoothness}). More detailed study is needed to quantify the interplay between dimension and smoothness to more specific bounds. 
  \end{remark}

  Theorem~\ref{thm:approximation} implies that any sufficiently smooth function can be approximated by an affine function, $a^\top x+b$, and a single-hidden-layer neural network with ReLU activations, with randomly generated weights and biases. In particular, the terms $a$, $b$, and $\cb_1,\ldots,\cb_m$, could be estimated via least-squares or stochastic gradient descent. 

  \begin{figure}
    \centering
    \includegraphics[width=.8\columnwidth]{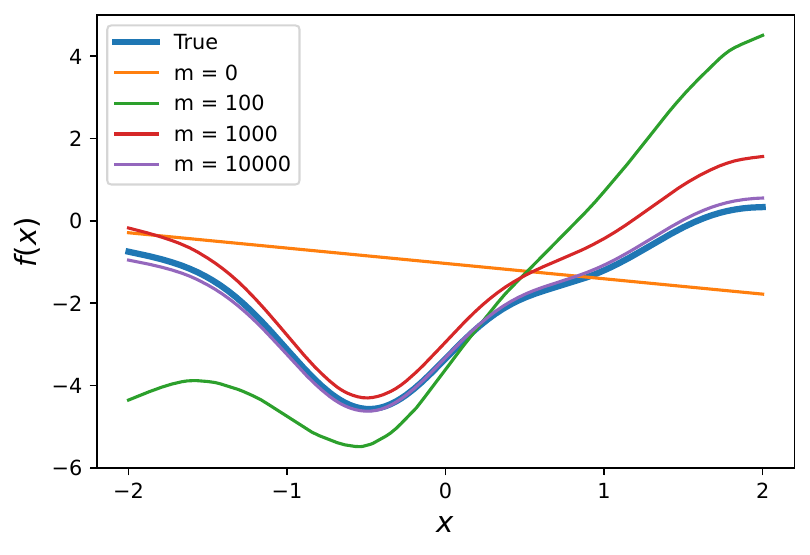}
    \caption{
      {
    \label{fig:importance}
        \bf Convergence of the Randomized Approximation.} The thick blue line shows the original function, $f$, while the thin lines show the neural network approximation $\fb_N$ for various numbers of neurons, $m$. When $m=0$, we have $\fb_N(x)=a^\top x +b$. The approximation scheme quickly approximates the general shape, which gets refined with increasing neurons. However, it should be noted that the coefficients, $a$, $b$, and $\cb_i$, are generated via the theoretical construction from the text. Better fits would be obtained if they were optimized. See Fig.~\ref{fig:optimized}.
    }
  \end{figure}
  
  Sampling over $\S^{n-1}$ can be achieved by sampling over $\R^n$ and normalizing the result. In particular, if $\wb$ is a Gaussian random variable with mean $0$ and identity covariance, then $\alphab=\frac{\wb}{\|\wb\|}$ is uniformly distributed over $\S^{n-1}$. 
  
  The rest of the section goes through a proof of Theorem~\ref{thm:approximation}. Here we provide an overview of the method.

  The main idea is to represent $f$ in the form:
  \begin{equation}
    \label{eq:expectationRepresention}
  f(x) = a^\top x + b + \E[g(\alphab,\tb)\sigma(\alphab^\top x -\tb)],
  \end{equation}
  where $\E$ denotes the expectation over $(\alphab,\tb)$, which are distributed according to $P$. The key step in deriving \eqref{eq:expectationRepresention} is the integral representation given in Subsection~\ref{ss:integral}. 
  
  Setting $\cb_i=\frac{g(\alphab_i,\tb_i)}{m}$, we have that the neural network component of \eqref{eq:fNN} is precisely the empirical mean:
  $$
  \sum_{i=1}^m\cb_i \sigma(\alphab_i^\top x-\tb_i)=\frac{1}{m}\sum_{i=1}^m g(\alphab_i,\tb_i)\sigma(\alphab_i^\top x -\tb)
  $$
  In Subsection~\ref{ss:importance}, we derive the error bound from Theorem~\ref{thm:approximation} by bounding the worst-case deviation of this empirical mean from its expected value.

    \begin{figure}
    \centering
    \includegraphics[width=.8\columnwidth]{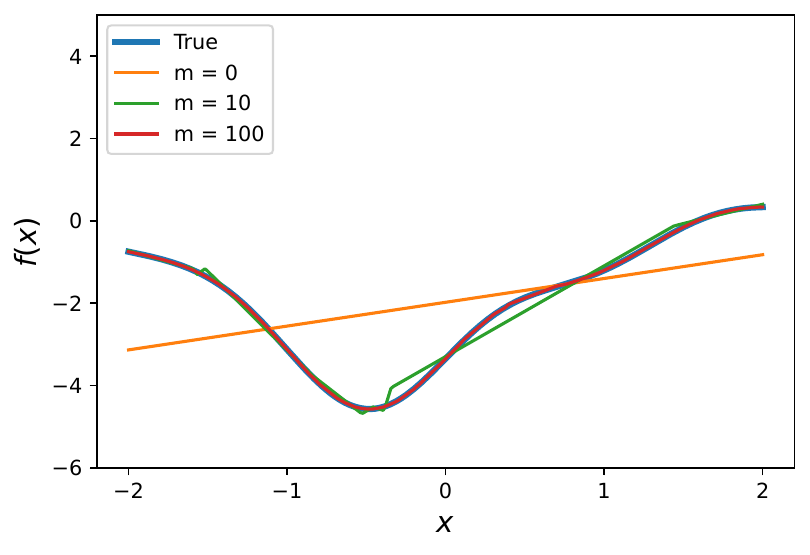}
    \caption{
    \label{fig:optimized}
      {\bf Optimized Fit.} Using the same weights and biases $(\alphab_i,\tb_i)$ used to generate Fig.~\ref{fig:importance}, we optimized the coefficients $a$, $b$, and $\cb_i$ via least-squares. With $100$ neurons, the optimized fit is nearly exact. 
    }
  \end{figure}

  \begin{remark}
    The coefficients $a$ and $b$ taken from the representation in \eqref{eq:expectationRepresention}, while $\cb_i$ come from random samples. They are not optimized, and so Theorem \ref{thm:approximation} gives a (potentially crude) upper bound on the achievable approximation error. Optimization can result in substantially better fits. See Fig.~\ref{fig:optimized}.
  \end{remark}
  
\subsection{Integral Representation}
\label{ss:integral}

Throughout this subsection, we will assume that $f:\R^n\to \R$ is such that 
\begin{align}
  \label{eq:Z}
  Z:=\int_{\R^n} |\hat f(\omega)| \|2\pi \omega\|^2 d\omega < \infty. 
\end{align}
As a result,
$$
p(\omega) = \frac{1}{Z} |\hat f(\omega)| \|2\pi \omega\|^2 
$$
is a probability density over $\R^n$ with $p(0)=0$.


Here  we derive an integral representation for a general class of real-valued functions. Related representations  were derived in \cite{petrosyan2020neural,sonoda2017neural}. Compared to those works, the representation below is only required to hold on a ball of radius $R$, while \cite{petrosyan2020neural,sonoda2017neural} give global representations. The advantage of the local representation below is that it enables precise bounding of the integrand. 

\begin{lemma}
  \label{lem:integral}
  {\it
    Assume that $f$ is real-valued, $\|\hat f\|_1<\infty$, and \eqref{eq:Z} holds. 
  There is a function $h:\S^{n-1}\times [-R,R]\to \R$ such that $\|h\|_\infty\le Z$, a vector $a\in\R^n$ with $\|a\|\le \sqrt{\|\hat f\|_1 Z}$, and a scalar $b\in\R$ with $|b|\le \|\hat f\|_1+R\sqrt{\|\hat f\|_1 Z}$ such that for all $\|x\|\le R$
  \begin{multline}
    \label{eq:represent}
    f(x)=\int_{\S^{n-1}}\int_{-R}^R h(\alpha,t)\sigma(\alpha^\top x-t) q(\alpha) dt \mu_{n-1}(d\alpha) \\
    +a^\top x +b,
  \end{multline}
  where $q$ is the probability density over $\S^{n-1}$ defined by
  \begin{equation}
    \label{eq:sphereDensity}
  q(\alpha) = \int_0^{\infty}p(r \alpha) r^{n-1} dr.
  \end{equation}
  }
\end{lemma}

\begin{proof}
  Let $\hat f(\omega)=e^{j2\pi \theta(\omega)}|\hat f(\omega)|$ be the magnitude and phase representation of the Fourier transform.
  Using the inverse Fourier transform gives
  \begin{align}
    \nonumber
    f(x)&=\int_{\R^n}e^{j2\pi (\omega^\top x+\theta(\omega))} |\hat f(\omega)|d\omega  \\
    \nonumber
        &=\int_{\R^n}\frac{Z}{\|2\pi\omega\|^2}e^{j2\pi (\omega^\top x+\theta(\omega))} p(\omega) d\omega \\
    \label{eq:cosIFT}
        &=\int_{\R^n} \frac{Z}{\|2\pi\omega\|^2} \cos(2\pi  (\omega^\top x+\theta(\omega)) p(\omega)d\omega.
  \end{align}
  The third equality follows because $f$ is real-valued.

  Define $\psi:\R\times (\R^n\setminus \{0\})\to\R$ by
  $$
  \psi(t,\omega)=\frac{Z}{\|2\pi \omega\|^2}\cos(2\pi(\|\omega\| t + \theta(\omega)))
  $$
  so that 
  \begin{equation}
    \label{eq:cosParameterized}
  \frac{Z}{\|2\pi\omega\|^2} \cos(2\pi  (\omega^\top x+\theta(\omega))=\psi\left(\left(\frac{\omega}{\|\omega\|}\right)^\top x,\omega \right). 
\end{equation}
Note that $\|x\|\le R$ implies that $\left|\left(\frac{\omega}{\|\omega\|}\right)^\top x \right|\le R$. 

  Direct calculation shows that for all $y\in [-R,R]$
  \begin{multline}
    \label{eq:scalarRepresentation}
    \psi(y,\omega)=\int_{-R}^R\frac{\partial^2\psi(t,\omega)}{\partial t^2}\sigma(y-t)dt\\
    +\frac{\partial \psi(-R,\omega)}{\partial t}(y+R)+\psi(-R,\omega). 
  \end{multline}
  A related identity was used  in the proof of Theorem 3 in \cite{breiman1993hinging}, but not written explicitly. 

  Combining \eqref{eq:cosIFT}, \eqref{eq:cosParameterized}, and \eqref{eq:scalarRepresentation} shows that for $\|x\|\le R$:
  \begin{multline}
    \label{eq:integralExpectation}
    f(x)=\left( \int_{\R^n} \frac{\partial \psi(-R,\omega)}{\partial t}\frac{\omega}{\|\omega\|}p(\omega)d\omega\right)^\top x \\
    + \int_{\R^n}\left(\frac{\partial \psi(-R,\omega)}{\partial t}R+\psi(-R,\omega)\right)p(\omega)d\omega\\
    +\int_{\R^n} \int_{-R}^R\frac{\partial^2\psi(t,\omega)}{\partial t^2}\sigma\left(\left(\frac{\omega}{\|\omega\|}\right)^\top x-t\right)p(\omega) dtd\omega.
  \end{multline}
  The first two lines define $a$ and $b$. 

  Now we show that all of the integrals in \eqref{eq:integralExpectation} converge absolutely. The integral of the $\psi(-R,\omega)$ term converges absolutely because $\|\hat f\|_1\le \infty$.
  
  The required derivatives of $\psi$ are given by:
  \begin{align*}
    \frac{\partial \psi(t,\omega)}{\partial t}&=-\frac{Z}{\|2\pi \omega\|}\sin(2\pi(\|\omega\| t+\theta(\omega)))\\
    \frac{\partial^2\psi(t,\omega)}{\partial t^2}&=-Z\cos(2\pi(\|\omega\|t + \theta(\omega))).
  \end{align*}

  The Cauchy-Schwarz inequality implies that
  $$
  \int_{\R^n}|\hat f(\omega)| \|2\pi \omega\| d\omega \le \sqrt{\|f\|_1}\sqrt{Z} <\infty,
  $$
  which implies that the integrals of the $\frac{\partial \psi(-R,\omega)}{\partial t}$ terms converge absolutely. In particular, we can bound $\|a\|\le \sqrt{\|\hat f\|_1}{\sqrt{Z}}$ and $|b|\le \|\hat f\|_1 + R \sqrt{\|\hat f\|_1 Z}$. 
  
  The third integral in \eqref{eq:integralExpectation} converges absolutely because $\left|\frac{\partial^2\psi(t,\omega)}{\partial t^2}\right|\le Z$. Absolute convergence of the third integral, in particular, enables the use of Fubini's theorem to switch the order of integration. 

  Now we derive the expression for $h$ from \eqref{eq:represent}. First we deal with the case that $n\ge 2$, and then discuss how the construction changes for $n=1$. 
  
  Let $\omega=r\beta(\phi)$ be the $n$-dimensional spherical coordinate representation from \cite{blumenson1960derivation}, where $r \ge 0$ and $\phi \in \Phi:= [0,\pi]^{n-2}\times [0,2\pi]$. As described in \cite{blumenson1960derivation}, the Jacobian determinant of the spherical coordinate representation has the form $r^{n-1}\zeta(\phi)$, there $\zeta(\phi)\ge 0$ for $\phi\in\Phi$.

  Now, since $p(\omega)$ is a probability density over $\R^n$, it must be that $p(r \beta(\phi))r^{n-1} \zeta(\phi)$ is a probability density over $[0,\infty)\times \Phi$. Using Bayes rule, we can factorize $p(r \beta(\phi))r^{n-1} \zeta(\phi)=\tilde q(r|\phi) \tilde q(\phi)$ where $\tilde q(\cdot|\phi)$ is a conditional density function on $[0,\infty)$ and $\tilde q(\cdot)$ is a density on $\Phi$. In particular, $\tilde q(\phi) = q(\beta(\phi))\zeta(\phi)$, where $q:\S^{n-1}\to \R$ was defined in the lemma statement. 
  
  If $\alpha = \beta(\phi)\in \S^{n-1}$, then the corresponding area element is given by $\mu_{n-1}(d\alpha) = \zeta(\phi)d\phi$, so that $q(\alpha)\mu_{n-1}(d\alpha) = \tilde q(\phi) d\phi$ defines a probability measure over $\S^{n-1}$, while $p(r\alpha )r^{n-1}dr \mu_{n-1}(d\alpha) = \tilde q(r|\phi) \tilde q(\phi)dr d\phi$ defines a probability measure over $[0,\infty)\times \S^{n-1}$.
  In particular, this implies that we can factorize $p(r \alpha)r^{n-1} = q(r|\alpha) q(\alpha)$, for a conditional density $q(\cdot|\alpha)$.

  We can express the third term on the right of \eqref{eq:integralExpectation} as:
  \begin{align}
    \nonumber
   \MoveEqLeft[0]
    \int_{\R^n} \int_{-R}^R\frac{\partial^2\psi(t,\omega)}{\partial t^2}\sigma\left(\left(\frac{\omega}{\|\omega\|}\right)^\top x-t\right)p(\omega) dtd\omega \\
    \label{eq:hInt}
    &=\int_{-R}^R\int_{\S^{n-1}}  h(\alpha,t)\sigma(\alpha^\top x-t)q(\alpha) \mu_{n-1}(d\alpha) dr
  \end{align}
  where
  \begin{equation*}
  \label{eq:hDef}
  h(\alpha,t)=\int_0^{\infty}\frac{\partial^2\psi(t,r \alpha )}{\partial t^2}q(r|\alpha)dr.
  \end{equation*}

  Now we describe the special case of $n=1$. As discussed in Subsection~\ref{ss:background}, $\S^{0}=\{-1,1\}$ and $\mu_{0}$ is the counting measure. Since $f$ is real-valued, we have that $|\hat f(\omega)|=|\hat f(-\omega)|$ for all $\omega\in\R$, so that $p(t)=p(-t)$ for all $t\ge 0$. As a result, \eqref{eq:sphereDensity} simplifies to give $q(-1)=q(1)=\frac{1}{2}$. In particular, $q(\alpha)\mu_0(d\alpha)$ indeed defines a probability distribution over $\S^{0}$. Furthermore, 
 $q(r|\alpha)=2p(r\alpha)=2p(r)$. With this interpretation of the $n=1$ case, we have that \eqref{eq:hInt} holds for all $n\ge 1$.

Returning to \eqref{eq:hInt}, the bound $\left\| \frac{\partial^2\psi}{\partial t^2}\right\|_{\infty}\le Z$ and the fact that $q(\cdot|\alpha)$ is a probability density over $[0,\infty)$ shows that $\|h\|_{\infty}\le Z$. 

  Plugging the definitions of $a$, $b$, and $h$ into \eqref{eq:integralExpectation} finishes the proof. 
\end{proof}

When $f$ is sufficiently smooth, we can get specific bounds on $h(\alpha,t)q(\alpha)$.

\begin{corollary}
  \label{cor:smallG}
  {\it
    If \eqref{eq:smoothness} holds, then \eqref{eq:represent} holds with $\sup_{(\alpha,t)\in \S^{n-1}\times [-1,1]}|h(\alpha,t)q(\alpha )|\le 8\pi^2 \rho$.
  }
\end{corollary}

\begin{proof}
  The smoothness assumption in \eqref{eq:smoothness} implies that for $i=0,1,2$:
  \begin{align}
    \nonumber
    \MoveEqLeft
    \int_{\R^n}|\hat f(\omega)| \|2\pi \omega\|^id\omega \le (2\pi)^i \rho \int_{\R^n}\frac{\|\omega\|^i}{1+\|\omega\|^k}d\omega \\
\nonumber
                                                       &=(2\pi)^i\rho \int_{\S^{n-1}} \int_0^{\infty} \frac{r^{n+i-1}}{1+r^k} dr \mu_{n-1}(d\alpha) \\
    \label{eq:smoothIntegral}
    &\le 2(2\pi)^i \rho A_{n-1},
  \end{align}
  where $A_{n-1}$ is the area of $\S^{n-1}$, from \eqref{eq:sphereArea}.

  In particular, we see that $\|\hat f\|_1<\infty$ and $Z<\infty$.  It suffices to bound $q(\alpha)$:
  \begin{align*}
    q(\alpha) &= \int_0^{\infty} \frac{1}{Z}|\hat f(r\alpha)| (2\pi)^2 r^{n+1} dr \\
              &\le \frac{4\pi^2 \rho}{Z} \int_0^{\infty} \frac{r^{n+1}}{1+r^k} dr \\
    &\le \frac{8\pi^2 \rho}{Z}.
  \end{align*}
  The result now follows because $|h(\alpha,t)|\le Z$. 
\end{proof}

\subsection{Importance Sampling}
\label{ss:importance}
Let $P$ be a probability density function over $\S^n\times [-R,R]$ with $\inf_{\S^n \times [-R,R]}P(\alpha,t)=P_{\min}>0$. We can turn the integral representation from \eqref{eq:represent} into a probabilistic representation via:
\begin{align}
  \nonumber
  \MoveEqLeft[0]
  f(x)-a^\top x-b\\
  \nonumber
  &=\int_{\S^{n-1}}\int_{-R}^R\frac{h(\alpha,t)q(\alpha)}{P(\alpha,t)}\sigma(\alpha^\top x-t) P(\alpha,t)dt \mu_{n-1}(d\alpha)  \\
  \label{eq:importance}
  &=\E\left[ \frac{h(\alphab,\tb)q(\alphab)}{P(\alphab,\tb)}\sigma(\alphab^\top x-\tb) \right],
\end{align}
where $\E$ denotes the expected value over $(\alphab,\tb)$ distributed according to  $P$. Note that \eqref{eq:importance} is a special case of \eqref{eq:expectationRepresention} with $g(\alpha,t)=\frac{h(\alpha,t)q(t)}{P(\alpha,t)}$.

Let $(\alphab_1,\tb_1),\ldots,(\alphab_m,\tb_m)$ be independent, identically distributed samples from $P$. The \emph{importance sampling} estimate of $f$ is defined by:
\begin{align*}
  \fb_I(x)=a^\top x + b + \frac{1}{m}\sum_{i=1}^m\frac{h(\alphab_i,\tb_i)q(\alphab_i)}{P(\alphab_i,\tb_i)}\sigma(\alphab_i^\top x-\tb_i).
\end{align*}

\begin{lemma}
  {\it
    If \eqref{eq:smoothness} holds, then for all $\nu \in (0,1)$, the following bound holds with probability at least $1-\nu$:
    \begin{multline*}
    \sup_{x\in B_0(R)}|\fb_I(x)-f(x)|\le \\\frac{1}{\sqrt{m}}\left(
      64 \sqrt{n\pi} LR+(\gamma+32LR)\sqrt{\log(4/\nu)}
    \right),
  \end{multline*}
  where
  \begin{subequations}
    \label{eq:importanceConstants}
  \begin{align}
    \gamma&=\frac{8\pi^2 \rho R}{P_{\min}}+4\rho A_{n-1}(1+R\pi) \\
    L&= \frac{8\pi^2 \rho}{P_{\min}}+8\pi A_{n-1} \rho.
  \end{align}
  \end{subequations}
  }
\end{lemma}
\begin{proof}
Define the random functions $\thetab$ and $\xib_i$ by
  \begin{align*}
    \thetab(x) &=\frac{1}{m}\sum_{i=1}^m\xib_i(x) \\
    \xib_i(x)&=\frac{h(\alphab_i,\tb_i)q(\alphab_i)}{P(\alphab_i,\tb_i)}\sigma(\alphab_i^\top x-\tb_i)
               +a^\top x +b
               - f(x).
  \end{align*}
  Lemma~\ref{lem:integral} implies that $\xib_i(x)$ have zero mean for all $\|x\|\le R$. Note that $\thetab(x)=\fb_I(x)-f(x)$.

  In order to bound $\sup_{x\in B_0(R)}|\thetab(x)|$, we utilize the following bound:
  \begin{align}
    \nonumber
    \sup_{x\in B_0(R)}|\thetab(x)| &= \sup_{x\in B_0(R)}|\thetab(x)-\thetab(0)+\thetab(0)| \\
    \label{eq:splitSup}
    &\le |\thetab(0)| + \sup_{x,y\in B_0(R)}|\thetab(x)-\thetab(y)|.
  \end{align}
  Then, we will derive bounds that hold with high probability for each term on the right. 

  First we bound $|\thetab(0)|$. To this end, we will show that $|\xib_i(0)|\le \gamma$, where $\gamma$ was defined in \eqref{eq:importanceConstants}. 

  The triangle inequality, followed by the bounds from Lemma~\ref{lem:integral},
  Corollary~\ref{cor:smallG}, and \eqref{eq:smoothIntegral} give
  \begin{align*}
    |\xib_i(0)|&\le \left|\frac{h(\alphab_i,\tb_i)q(\alphab_i)}{P(\alphab_i,\tb_i)} t_i\right| +|b|+|f(0)| \\
               &\le \frac{8\pi^2 \rho R}{P_{\min}}+\|\hat f\|_1+R\sqrt{\|\hat f\|_1 Z}+\|\hat f\|_1 \\
    &\le \frac{8\pi^2 \rho R}{P_{\min}}+4\rho A_{n-1}(1+R\pi)
  \end{align*}
  
  A random variable, $\vb$, is called $\sigma$-sub-Gaussian if $\E[e^{\lambda \vb}]\le e^{\frac{\lambda^2 \sigma^2}{2}}$ for all $\lambda \in \R$. Using Hoeffding's lemma and the bound $|\xib_i(0)|\le \gamma$ shows that $\xib_i(0)$ is $\gamma$-sub-Gaussian. Independence of $\xib_i(0)$ shows that
  \begin{equation}
    \label{eq:indepSubgauss}
    \E[e^{\lambda\thetab(0)}]=\prod_{i=1}^m\E[e^{\lambda \xib_i(0)/m}]\le \exp\left(\frac{\lambda^2 \gamma^2}{2m}\right),
  \end{equation}
  so that $\thetab(0)$ is $\frac{\gamma}{\sqrt{m}}$-sub-Gaussian. Equation (2.9) of \cite{wainwright2019high} implies that for all $t\ge 0$
  $$
  \P(|\thetab(0)|\ge t)\le 2 \exp\left(-\frac{mt^2}{2\gamma} \right).
  $$

  Setting $2 \exp\left(-\frac{mt^2}{2\gamma^2} \right)=\frac{\nu}{2}$ gives
  \begin{equation}
    \label{eq:theta0Error}
    \P\left(|\thetab(0)|\ge \gamma \sqrt{\frac{2\log(4/\nu)}{m}}
    \right)\le \frac{\nu}{2}.
  \end{equation}

  Now we turn to bounding $\sup_{x,y\in B_0(R)}|\thetab(x)-\thetab(y)|$. 

  We show that  
  $\xib_i$ are $L$-Lipschitz, where $L$ was defined in \eqref{eq:importanceConstants}.
  Corollary~\ref{cor:smallG} imply that $\left|\frac{h(\alphab_i,\tb_i)q(\alphab_i)}{P(\alphab_i,\tb_i)} \right|\le \frac{8\pi^2 \rho}{P_{\min}}$. The bounds from \eqref{eq:smoothIntegral} implies that $f$ is  $4\pi A_{n-1}\rho$-Lipschitz and that $\|a\|\le \sqrt{\|\hat f\|_1 Z} \le 4\pi A_{n-1}\rho$. 
  The bound on $L$ now follows via the triangle inequality, since $\sigma$ is $1$-Lipschitz. 

  Now, we will bound $\sup_{x\in B_0(R)}|\thetab(x)|$ with high probability using the Dudley entropy integral, using the methodology from \cite{wainwright2019high}. Some background is required to describe the result. 
   Let $\psi_2(t) = e^{t^2}-1$. The corresponding \emph{Orlicz} norm for a zero-mean random scalar variable, $\vb$, is defined by
  $$
  \|\vb\|_{\psi_2}=\inf\{\lambda > 0 | \E[\psi_2(\vb/\lambda)]\le 1\}.
  $$
If $\vb$ is $\sigma$-sub-Gaussian, then $\|\vb\|_{\psi_2}\le 2\sigma$. (This is shown in Lemma~7 of \cite{lamperski2023nonasymptotic}, using a variation on the proof of Proposition 2.5.2 from \cite{vershynin2018high}.)

  For $x,y\in B_0(R)$, we have that $|\xib_i(x)-\xib_i(y)| \le L\|x-y\|$. And so Hoeffding's Lemma implies that $\xib_i(x)-\xib_i(y)$ is $L\|x-y\|$-sub-Gaussian. Similar to \eqref{eq:indepSubgauss}, we have that $\thetab(x)-\thetab(y)$ is $\frac{L\|x-y\|}{\sqrt{m}}$-sub-Gaussian. Thus, 
  $$
  \|\thetab(x)-\thetab(y)\|_{\psi_2}\le \frac{2L}{\sqrt{m}}\|x-y\|
  $$

  This shows that $\thetab$ is an \emph{Orlicz process} with respect to the metric $d(x,y)=\frac{2L}{\sqrt{m}}\|x-y\|$. The diameter of $B_0(R)$ with respect to $d$ is $D:=\frac{4L R}{\sqrt{m}}$. 

  Let $N(\epsilon,B_0(R),d)$ denote the $\epsilon$-\emph{covering number} of $B_0(R)$ under the metric $d$, which corresponds to the minimal cardinality of a covering of $B_0(R)$ with $d$-balls of radius $\epsilon$. For any $\epsilon>0$, $N(\epsilon,B_0(1),\|\cdot\|)\le \left(1+\frac{2}{\epsilon}\right)^n$. By rescaling the coverings of the Euclidean unit ball, we get the bound $N(\epsilon,B_0(R),d)\le \left(1+\frac{D}{\epsilon}\right)^n$, for all $\epsilon >0$. 

  Let $\Jc=\int_0^D\sqrt{\log N(\epsilon,B_0(R),d)}d\epsilon$. 
  Theorem~5.36 of \cite{wainwright2019high} implies that the following probability bound holds for all $t>0$:
  \begin{align*}
    \P\left(
      \sup_{x,y\in B_0(R)}|\thetab(x)-\thetab(y)| \ge
      8(t+\Jc)
    \right)\le 2e^{-t^2/D^2}.
  \end{align*}
  (Theorem~5.36 in \cite{wainwright2019high} is proved for a general class of Orlicz norms and stated with an unspecified constant. When specialized to $\|\cdot\|_{\psi_2}$, the corresponding constant is $8$.)

  Setting $2e^{-t^2/D^2}=\nu/2$ gives
  \begin{multline}
  \label{eq:dudley}
     \P\left(
      \sup_{x,y\in B_0(R)}|\thetab(x)-\thetab(y)| \ge
      8\left(\Jc+D\sqrt{\log(4/\nu)}\right)
    \right)\\\le \frac{\nu}{2}.
\end{multline}
  
  Now, we bound $\Jc$ via some integral manipulations:
  \begin{align*}
    \Jc&\le \sqrt{n} \int_0^D\sqrt{\log\left(1+\frac{D}{\epsilon}\right)}d\epsilon \\
       &\overset{t=1+\frac{D}{\epsilon}}{=}\sqrt{n}D\int_2^{\infty} \frac{\sqrt{\log t}}{(t-1)^2}dt \\
       &\le 2\sqrt{n}D \int_2^{\infty} \frac{\sqrt{\log t}}{t^2}dt \\
       &\overset{u=\log t}{=}2\sqrt{n}D\int_{\log 2}^{\infty} u^{-1/2} e^{-u}du \\
    &\overset{\Gamma(1/2)=\sqrt{\pi}}{\le} 2\sqrt{n\pi}D.
  \end{align*}

  Now, we combine the analysis of $\thetab(0)$ and the worst-case difference. A union bound shows that
  \begin{multline*}
    \P\left( |\thetab(0)|\ge \gamma \sqrt{\frac{2\log(4/\nu)}{m}} \cup \right.\\
    \left.\sup_{x,y\in B_0(R)}|\thetab(x)-\thetab(y)| \ge
      8\left(\Jc+D\sqrt{\log(4/\nu)}\right) \right) \le \nu
  \end{multline*}
  De Morgan's laws then show that with probability at least $1-\nu$, both of the following inequalities hold:
  \begin{align*}
    |\thetab(0)|&< \gamma \sqrt{\frac{2\log(4/\nu)}{m}}  \\
    \sup_{x,y\in B_0(R)}|\thetab(x)-\thetab(y)| &<\
      8\left(\Jc+D\sqrt{\log(4/\nu)}\right). 
  \end{align*}
  The result now follows from \eqref{eq:splitSup}, after plugging in the values for $\Jc$ and $D$.
\end{proof}

\subsection{Proof of Theorem~\ref{thm:approximation}}
The proof follows by setting $\cb_i=\frac{h(\alphab_i,\tb_i)q(\alphab_i)}{mP(\alphab_i,\tb_i)}$ and collecting all of the bounds on the terms. 
\hfill\QED


%% file: application.tex
\section{Application}
\label{sec:apx:application}

In this section, we apply Theorem~\ref{thm:approximation} to an approximate Model Reference 
Adaptive Control (MRAC) method from \cite{lavretsky2013robust}.

\subsection{Setup and Existing Results}
Chapter 12 of \cite{lavretsky2013robust} describes an approximate MRAC scheme for plants of the form 
\begin{align}
\label{eq:apx:plant}
\dot{x}_t \ &= \ A\,x_t + B\big(u_t + f(x)), 
\end{align}
where $f:\R^n\to\R^\ell$ is an unknown nonlinearity. (The setup in \cite{lavretsky2013robust} is more general, as it also includes a disturbance term and an unknown multiplier matrix on $B$.)

We will just provide an overview of the methodology, and describe how Theorem~\ref{thm:approximation} can be used to give bounds on controller performance, while deferring the details of the controller and its analysis to \cite{lavretsky2013robust}.

In \eqref{eq:apx:plant}, $A$ is an unknown $n\times n$ state matrix for the plant state $x_t\in\R^n$, $B$ is a known
$n\times\ell$ input matrix for the input $u_t\in\R^\ell$.

It is \emph{assumed} that there is a nonlinear vector function $\Psi:\R^n\to\R^N$ and an \underline{unknown} $\ell\times N$ parameter matrix, $\Theta$ such that $\Theta \Psi(x)$ gives a good approximation to $f$ on a bounded region. The precise details are described in Theorem~\ref{thm:control}, below.

It is assumed that there exists an $n\times\ell$ matrix of feedback gains $K_x$ and an
$\ell\times\ell$ matrix of feedforward gains $K_r$ satisfying the \textit{matching conditions}
\begin{align*}\nonumber
A + BK_x = A_r \\
BK_r = B_r
\end{align*}
to a controllable, linear reference model
\begin{equation}
  \label{eq:reference}
\dot{x}_t^r = A_r\,x_t^r + B_r\,r_t ,
\end{equation}
where $A_r$ is a known Hurwitz $n\times n$ reference state matrix for the reference state
$x^r_t\in\R^n$, $B_r$ is a known $n\times\ell$ reference input matrix, and $r_t\in\R^\ell$ is a
bounded reference input. 


The adaptive law takes the form
\begin{equation}
  \label{eq:controller}
u_t \ = \hat K_{x,t} x_t - \Thetah_t\!\Psi(x_t) + \left(1-\mu(x_t)\right)\hat K_{r,t} r_t + 
\mu(x_t)\uw(x_t),
\end{equation}
where the dynamics of $\hat K_{x,t}$, $\Thetah_t$, and $\hat K_{r,t}$ are designed via Lyapunov methods, $\uw$ is a control law that drives the state back to a compact set if it ever gets too far, and $\mu$ is a weighting function that balances the contributions of the feedforward term $\hat K_{r,t}^\top r_t$, and the controller $\uw(x_t)$. 

Let $P$ and $Q$ be positive definite matrices such that the Lyapunov equation holds:
$$
A_r^\top P+PA_r = -Q.
$$

The following result summarizes the  convergences properties of the controller from Chapter 12 of \cite{lavretsky2013robust}.
\begin{theorem}[\cite{lavretsky2013robust}]
  \label{thm:control}
  {\it
Let $x_t^r$ be a fixed reference trajectory generated by a bounded reference input $r_t$.
Assume that there is an \underline{unknown} parameter matrix $\Theta$, a known bounding function $\epsilon_{\max}$, and positive numbers $R$ and $\epsilon_0$ such that the following conditions hold:
\begin{enumerate}[(i)]
\item $\|f(x)-\Theta\Psi(x)\|\le \epsilon_{\max}(x)$ for all $x\in \R^n$
\item $R\ge 4 \|P\|\|Q^{-1}\|\epsilon_0 + \sup_{t\ge 0} \|x_t^r\|$
\item $\sup_{x\in B_0(R)}\|f(x)-\Theta\Psi(x)\|\le \epsilon_0$
\end{enumerate}
Then there is an adaptive law of the form in \eqref{eq:controller} such that for any initial condition, $x_0$:
    \begin{itemize}
    \item There is a time $T_1$ such that $x_t \in B_0(R)$  for all $t\ge T_1$.
    \item There is a time $T_2$ such that the tracking error satisfies $(x_t-x_t^r)\in B_0\left(4 \|P\|\|Q^{-1}\|\epsilon_0\right)$ for all $t\ge T_2$.
    \end{itemize}
  }
\end{theorem}

In particular, Theorem~\ref{thm:control} states that as long as a good approximation of the form $\Theta \Psi(x)$ exists over a sufficiently large bounded region, the controller will drive the state to a bounded region and make the tracking error arbitrarily small. We do not actually need to know the parameter matrix, $\Theta$.

A gap in existing adaptive control analysis with such linear parametrizations is \emph{proving} that the $\Theta$ matrix exists.

In \cite{lavretsky2013robust} it is suggested that the entries of $\Psi$ take the form $\Psi_i(x)=\phi(\alpha_i^\top x-b_i)$, where
$\phi$ is a neural network activation function, $\alpha_i$ is a weight vector, and $b_i$ is a bias. Classical non-constructive approximation theorems guarantee that a collection of weights and biases, $(\alpha_i,b_i)$, and a parameter matrix, $\Theta$, \emph{exist}
such that $\sup_{x\in B_0(R)}\|f(x)-\Theta \Psi(x)\|\le \epsilon_0$, but do not describe how to find them. So, there is no guarantee (without further work, as described below) that a suitable $\Theta$ exists for a given $\Psi$.

While we focus on a method from \cite{lavretsky2013robust} for concreteness, 
similar gaps in the analysis are common in other works on control with neural networks \cite{lavretsky2013robust,vrabie2013optimal,kamalapurkar2018reinforcement,greene2023deep,makumi2023approximate,cohen2023safe,kokolakis2023reachability,sung2023robust,lian2024inverse}. These gaps are specifically articulated in \cite{soudbakhsh2023data,annaswamy2023adaptive}.


\subsection{Guaranteed Approximations for the Nonlinearity}

The result below gives a randomized construction of the nonlinear vector function, $\Psib$, such that for any $\epsilon_0>0$  and any $R>0$, there is a matrix $\Thetab$ such that $\sup_{x\in B_0(R)}\|f(x)-\Thetab \Psib(x)\|\le \epsilon_0$ with high probability. Furthermore, it quantifies the required dimension of $\Psib$ and gives a specific expression for the error bounding function $\epsilon_{\max}$. The result is that as long as $f$ is sufficiently smooth, we have \emph{proved} that (with high probability) $\Thetab \Psib(x)$ is a sufficiently good approximation to apply Theorem~\ref{thm:control},  rather than assumed it.   

\begin{lemma}
  \label{lem:vectorBoundsforAC}
  {
    \it
    Assume that every entry of $f:\R^n\to \R^\ell$ satisfies \eqref{eq:smoothness}.
    Define $\Psib:\R^n\to \R^{m+n+1}$ by:
    \begin{align*}
    \Psib(x)^\top&=\begin{bmatrix}
      1 & x^\top & \sigma(\alphab_1^\top x-\tb_1) & \cdots & \sigma(\alphab_{m}^\top x-\tb_m)
    \end{bmatrix}
    \end{align*}
      where $(\alphab_1,\tb_1),\ldots,(\alphab_m,\tb_m)$ are independent identically distributed samples from $P$. If
      $$
      m \ge \frac{\ell}{\epsilon_0^2}\left(\kappa_0+\kappa_1 \sqrt{\log(4\ell/\nu)} \right)^2,
      $$
      then for any $\nu \in (0,1)$, with probability at least $1-\nu$, there is a matrix $\Thetab$ such that $\sup_{x\in B_0(R)}\|f(x)-\Thetab \Psib(x)\| \le \epsilon_0$. Furthermore, the bounding function can be taken as:
      \begin{multline*}
        \epsilon_{\max}(x)=2\sqrt{\ell}\rho A_{n-1}+ \sqrt{\ell}\rho(1+\|x\|\sqrt{m+1})\cdot \\
        \left((1+4\pi+2\pi R) A_{n-1} + \frac{8\pi^2\rho}{\sqrt{m}P_{\min}}\right).
      \end{multline*}
  }
\end{lemma}

\begin{proof}
  Let $\thetab_i(x)$ be the $i$th entry of $\Thetab \Psib(x) - f(x)$, where the entries of $\Thetab$ are constructed from the importance sampling approximation for each $f_i(x)$. For each $i$, we have that the following error
  $$
  \sup_{x\in B_0(R)}|\thetab_i(x)| \ge \frac{1}{m}\left(\kappa_0+\kappa_1 \sqrt{4\ell /\nu}\right),
  $$
  holds with probability at most $\nu/\ell$. So, a union bounding / De Morgan argument shows that with probability at least $1-\nu$, all entries satisfy
  $$
  \sup_{x\in B_0(R)}|\thetab_i(x)| \le \frac{1}{\sqrt{m}}\left(\kappa_0+\kappa_1 \sqrt{4\ell /\nu}\right),
  $$
  which implies that
  $$
  \| \Thetab \Psib(x) - f(x)\|\le \frac{\sqrt{\ell}}{\sqrt{m}}\left(\kappa_0+\kappa_1 \sqrt{4\ell /\nu}\right).
  $$
  The sufficient condition for $\|\Thetab\Psib(x)-f(x)\|\le \epsilon_0$  now follows by re-arrangement.

  The bound on $\epsilon_{\max}(x)$ uses the triangle inequality:
  \begin{equation*}
    \|f(x)-\Thetab\Psi(x)\|\le \|f(x)\|+\|\Thetab\|\|\Psi(x)\|.
  \end{equation*}
  Then we  bound $\|f_i\|_{\infty} \le \|\hat f_i\|_1\le 2\rho A_{n-1}$ , $\|\Psib(x)\|\le (1+\|x\|\sqrt{m+1})$, and use the bounds on the coefficients from Theorem~\ref{thm:approximation} to bound $\|\Thetab\|$. 
\end{proof}


%% file: conclusion.tex
\section{Conclusion and Future Directions}
\label{sec:conclusion}

In this paper, we gave a simple bound on the error in approximation smooth functions with random ReLU networks. We showed how the results can be applied to an adaptive control problem. The key intermediate result was a novel integral representation theorem for ReLU activation functions. Remaining theoretical challenges include quantifying the effects of smoothness in high dimensions (i.e. the scaling of $\rho A_{n-1}$), and relaxing the smoothness requirements. Natural extensions include the examination of other activation functions and applications to different control problems. Other interesting directions would be extensions to deep networks and networks with trained hidden layers.
